\documentclass[a4paper,12pt]{amsart}
\pdfoutput=1

\usepackage[sc]{mathpazo}
\usepackage[T1]{fontenc}

\usepackage{color}
\definecolor{red}{rgb}{1,0,0}

\definecolor{gre}{rgb}{0,0.7,0}

\usepackage{graphicx}
\usepackage{amsmath}
\usepackage{amssymb}
\usepackage[active]{srcltx}
\usepackage{hyperref}

\newtheorem{prop}{Proposition}%
\theoremstyle{definition}

\theoremstyle{remark}
\theoremstyle{plain}

\def\NN{{\mathbb N}}

\def\RR{{\mathbb R}}

\def\ZZ{{\mathbb Z}}
\def\Z{{\mathbb Z}}

\def\vecb{{\text{\boldmath$b$}}}

\def\vece{{\text{\boldmath$e$}}}

\def\vecr{{\text{\boldmath$r$}}}

\def\vecs{{\text{\boldmath$s$}}}

\def\vecu{{\text{\boldmath$u$}}}

\def\vecx{{\text{\boldmath$x$}}}

\def\vecnull{{\text{\boldmath$0$}}}

\def\scrA{{\mathcal A}}

\def\scrJ{{\mathcal J}}

\def\scrL{{\mathcal L}}

\def\SL{\operatorname{SL}}

\def\GamG{\Gamma\backslash G}

\def\SLZ{\SL(2,\ZZ)}
\def\SLR{\SL(2,\RR)}

\begin{document}

\title{The three gap theorem and the space of lattices}
\author{Jens Marklof}
\author{Andreas Str\"ombergsson}
\address{School of Mathematics, University of Bristol,
Bristol BS8 1TW, U.K.\newline
\rule[0ex]{0ex}{0ex} \hspace{8pt}{\tt j.marklof@bristol.ac.uk}}
\address{Department of Mathematics, Box 480, Uppsala University,
SE-75106 Uppsala, Sweden\newline
\rule[0ex]{0ex}{0ex} \hspace{8pt}{\tt astrombe@math.uu.se}}\date{\today}
\thanks{To appear in the American Mathematical Monthly}
\subjclass[2010]{11H06; 11J71, 52C05}
\date{14 December 2016/21 June 2017}

\begin{abstract}
The three gap theorem (or Steinhaus conjecture) asserts that there are at most three distinct gap lengths in the fractional parts of the sequence $\alpha,2\alpha,\ldots,N\alpha$, for any integer $N$ and real number $\alpha$. This statement was proved in the 1950s independently by various authors. Here we present a different approach using the space of two-dimensional Euclidean lattices. 
\end{abstract}

\maketitle

Imagine we divide a cake by cutting a first wedge at an angle $\alpha$, then an identical second, third,  and so on as illustrated in Figure \ref{fig0} (left), until the remaining piece is either of the same size as the previous, or smaller. We now have a cake comprising wedges of at most two distinct sizes: the size of the original and that of the left-over wedge. Suppose we continue cutting but insist that after each cut we rotate the knife by the same angle $\alpha$ as before, see Figure \ref{fig0} (right). How many different sizes of cake wedges are there after $N$ cuts? The celebrated ``three gap theorem'' states that for each $N$ there will be at most three! This surprising fact was understood by number theorists in the late 1950s \cite{Sos57,Sos58,Suranyi58,Swierczkowski59}. Various new proofs have appeared since then, with connections to continued fractions \cite{Slater,Ravenstein}, Riemannian geometry \cite{BS} and elementary topology \cite[App.~A]{Haynes}, as well as higher-dimensional generalisations \cite{Bleher,Chevallier,Vijay}. Our aim here is to provide a simple proof of the three gap phenomenon by exploiting the geometry of the space of two-dimensional Euclidean lattices. 

\begin{figure}[h]
\begin{center}
\begin{minipage}{1.0\textwidth}
\unitlength0.1\textwidth
\begin{picture}(8,3.1)(0,0)
\put(1.2,0.1){\includegraphics[width=0.3\textwidth]{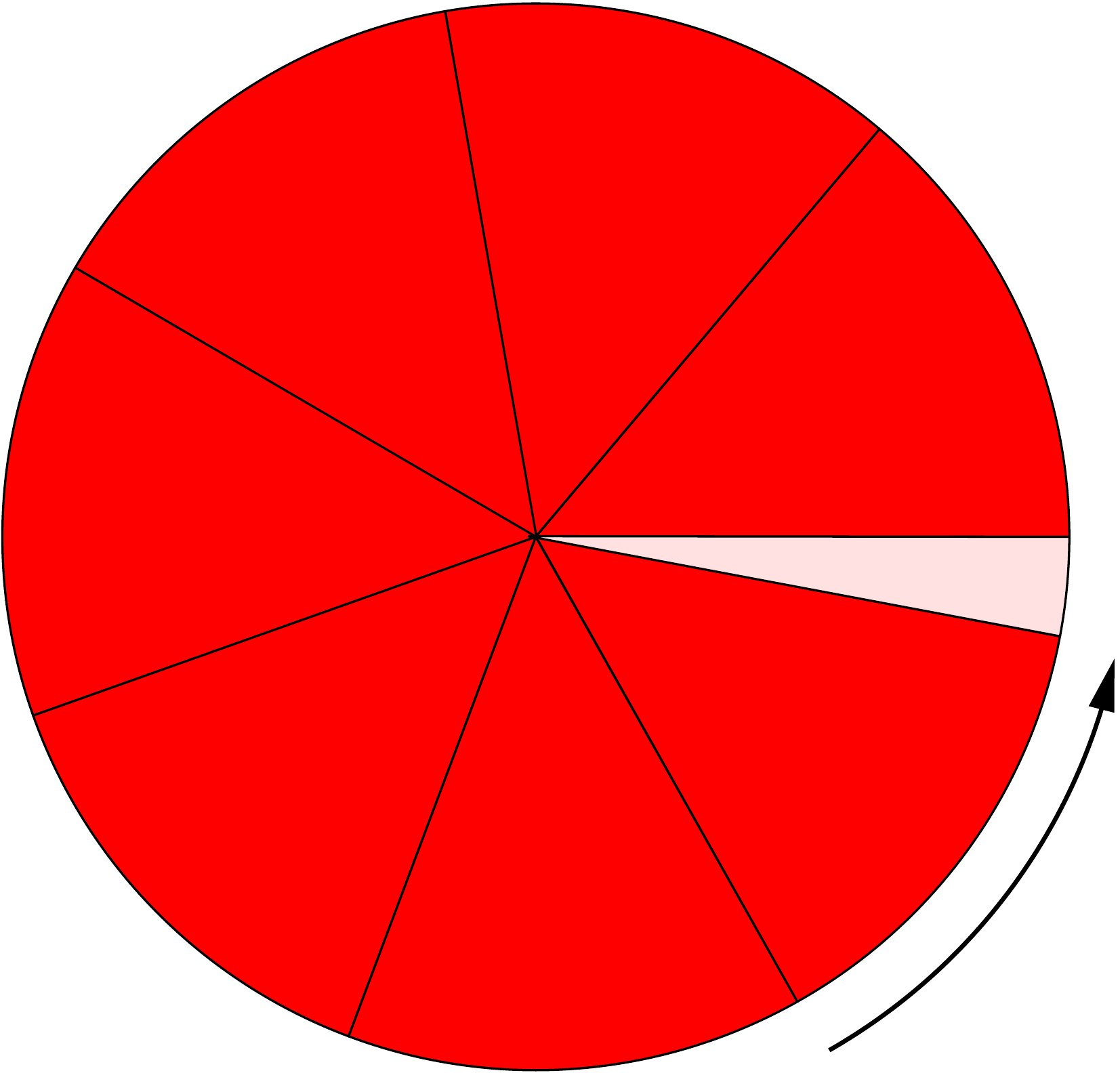}}
\put(5.4,0.1){\includegraphics[width=0.3\textwidth]{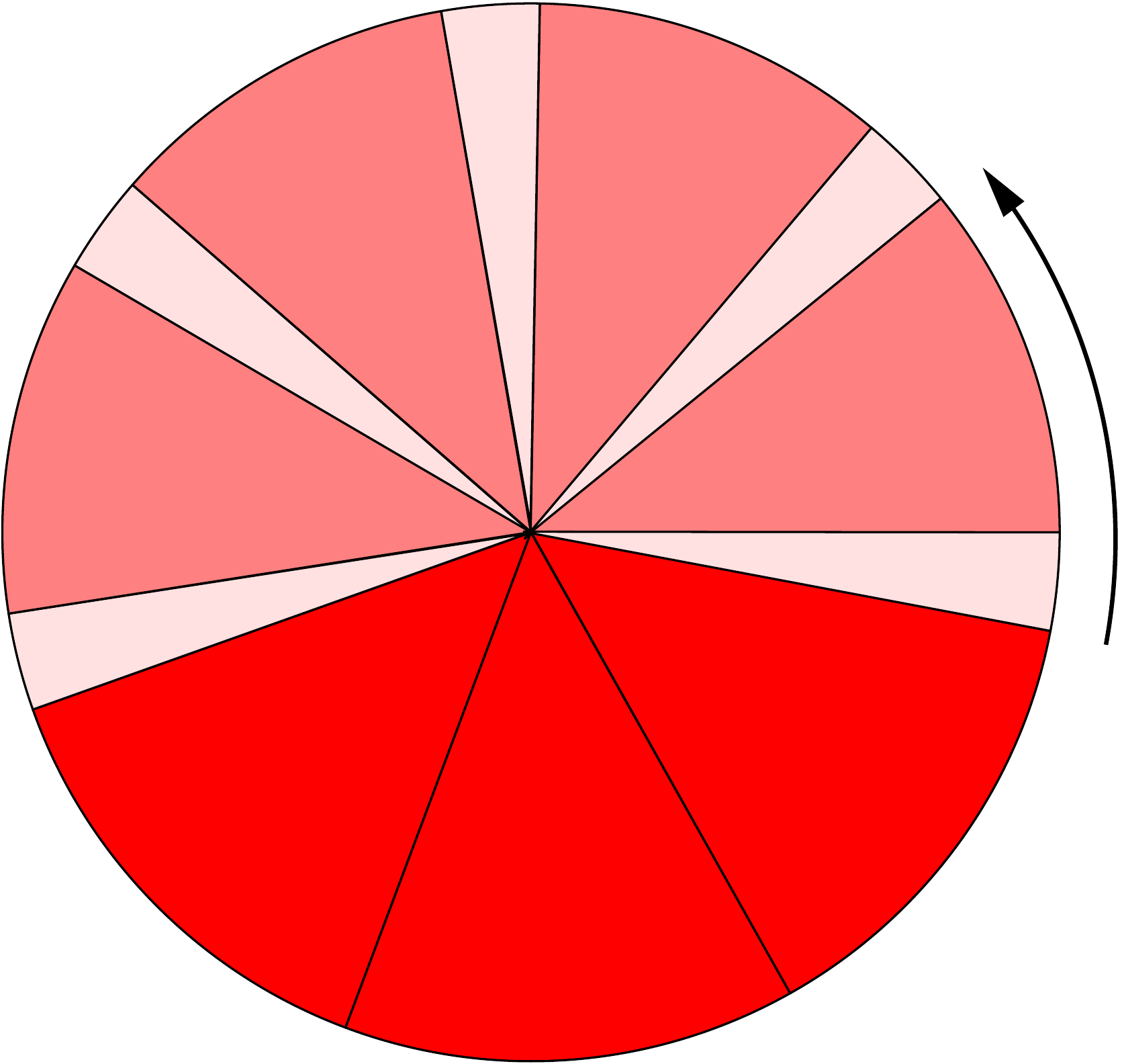}}
\put(4.05,0.6){$\alpha$} 
\put(8.35,2.1){$\alpha$} 
\end{picture}
\end{minipage}
\end{center}
\caption{For each given $N$, there are at most three different wedge sizes.} \label{fig0}
\end{figure}

The standard example of a Euclidean lattice in $\RR^2$ is the square lattice $\ZZ^2$. We can generate any other Euclidean lattice $\scrL$ in $\RR^2$ by applying a linear transformation to $\ZZ^2$. Writing points in $\RR^2$ as row vectors $\vecx=(x_1,x_2)$, we have explicitly
\begin{equation}
\scrL = \ZZ^2 M=\{(m,n)M\mid (m,n)\in\ZZ^2\} ,
\end{equation}
where $M$ is a $2\times 2$ matrix with real coefficients. If 
\begin{equation}
M=\begin{pmatrix} a & b \\ c & d \end{pmatrix} , \qquad \det M=ad-bc\neq 0,
\end{equation}
then a basis of the lattice $\scrL=\ZZ^2 M$ is given by the linearly independent vectors
\begin{equation}
\vecb_1 = \vece_1 M = (a,b), \qquad \vecb_2 = \vece_2 M = (c,d),
\end{equation}
where $\vece_1=(1,0)$,  $\vece_2=(0,1)$ is the standard basis of $\ZZ^2$. All other bases of $\scrL$ with the same orientation can be obtained by replacing $M$ by $\gamma M$ provided $\gamma\in\Gamma=\SLZ$, the group of matrices with integer coefficients and unit determinant. 
In the following we restrict our attention to lattices $\scrL=\ZZ^2 M$ whose basis vectors span a parallelogram of unit area. This means that $\det M = \pm1$, and by reversing the orientation of a basis vector where necessary (this will not change the lattice), we can assume in fact that $\det M=1$.
Let us therefore denote by $G=\SLR$ the group of real matrices with unit determinant.
The ``modular group'' $\Gamma=\SLZ$ is a discrete subgroup of  $G$, and the space of lattices can in this way be identified with the coset space $\GamG=\{ \Gamma g \mid g\in G\}$. 

In order to translate the three gap problem into the setting of lattices, let us measure all angles in units of $360^\circ$. That is, angles are parametrized by the coset space $\RR/\ZZ=\{ x+\ZZ \mid x\in\RR\}$ (the set of reals taken modulo one), which we can think of as the unit interval $[0,1]$ with the endpoints $0$ and $1$ identified. Fix $\alpha\in\RR/\ZZ$, and let $\xi_k=\{ k\alpha \}$ be the fractional part of $k\alpha$. The quantity $\xi_k$ represents the angular position of the $k$th cut. The angles of the resulting cake wedges after $N$ cuts are precisely the gaps between the elements of the sequence $(\xi_k)_{k=1}^{N}$ on $\RR/\ZZ$. These gaps are, in other words, the lengths of the intervals that $\RR/\ZZ$ is partitioned into by $(\xi_k)_{k=1}^{N}$. 

The gap between $\xi_k$ and its {\em next} neighbor on $\RR/\ZZ$ (this is not necessarily the {\em nearest} neighbor, as the gap to the element preceding $\xi_k$ may be the smaller one) is given by 
\begin{equation}
s_{k,N}  = \min\{ (\ell-k)\alpha + n > 0 \mid (\ell,n)\in\ZZ^2,\; 0< \ell \leq N  \} .
\end{equation}
The substitution $m=\ell-k$ yields
\begin{equation}\label{0st}
s_{k,N}  = \min\{ m \alpha + n  >  0 \mid (m,n)\in\ZZ^2,\; -k< m \leq N-k \} .
\end{equation}
We rewrite \eqref{0st} as
\begin{equation}\label{resca}
s_{k,N} = \min\{ y > 0 \mid (x,y)\in\ZZ^2 A_1 , \; -k< x \leq N-k \},
\end{equation}
with the matrix
\begin{equation}\label{lattice0}
A_1=\begin{pmatrix} 1 & \alpha \\ 0 & 1 \end{pmatrix} .
\end{equation}
The lattice $\ZZ^2 A_1$ and $s_{k,N}$ are illustrated in Figure \ref{fig001}.

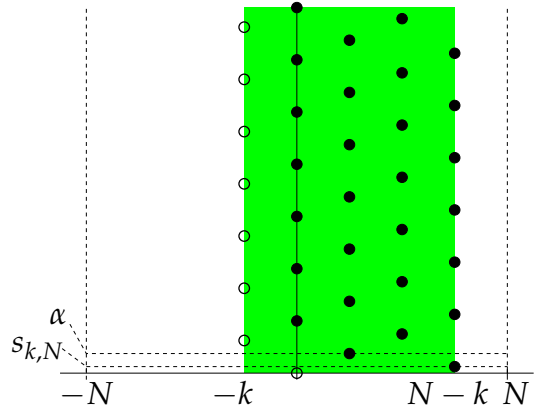
\begin{figure}
\begin{center}
\begin{minipage}{0.4\textwidth}
\unitlength0.1\textwidth
\begin{picture}(10,10)(0,0)
\put(0,1){\includegraphics[width=\textwidth]{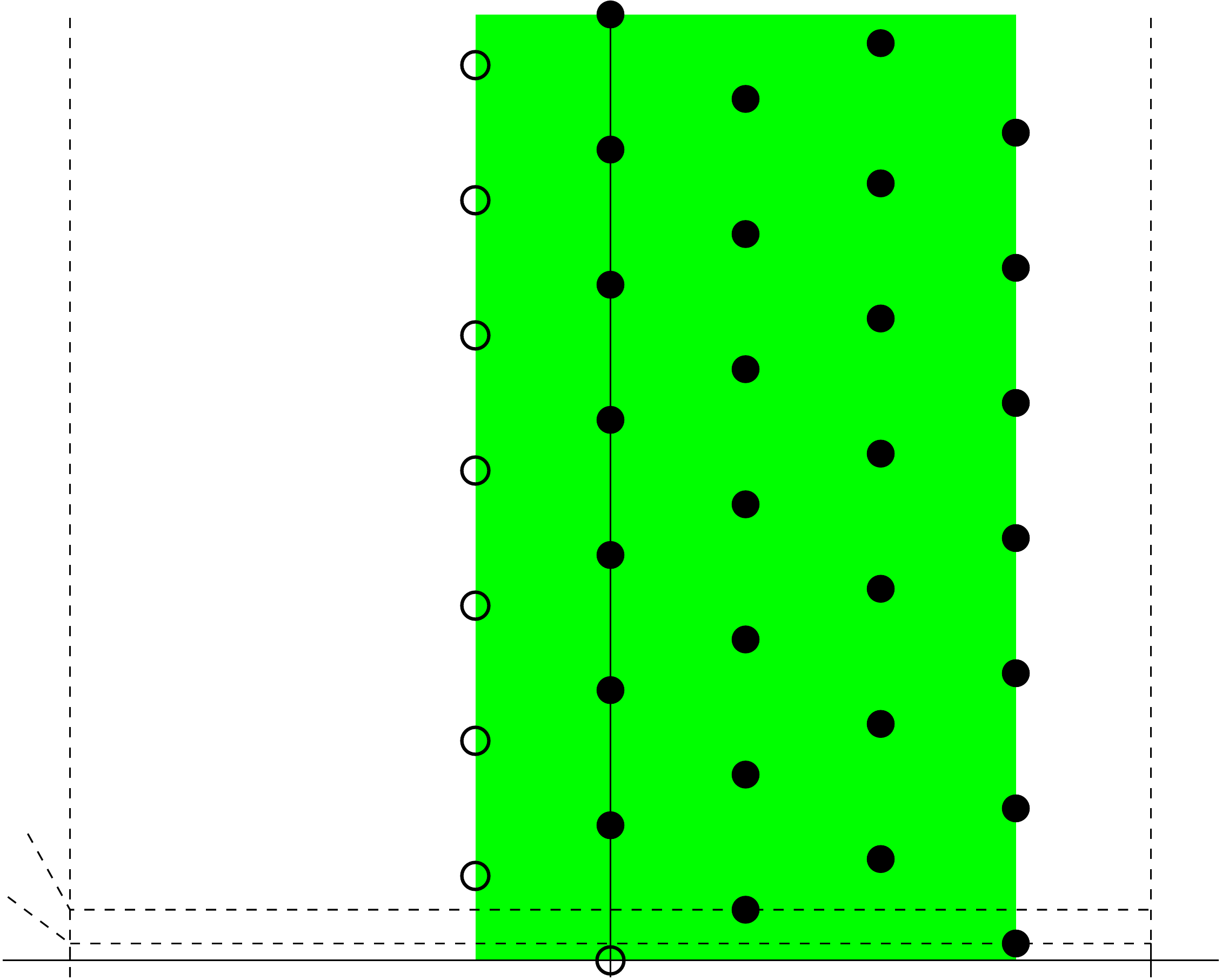}}
\put(-1,1.6){$s_{k,N}$}\put(-0.2,2.2){$\alpha$}  
\put(0,0.5){$-N$}\put(3.2,0.5){$-k$} \put(7.3,0.5){$N-k$} \put(9.3,0.5){$N$} 
\end{picture}
\end{minipage}
\end{center}
\caption{Illustration of the the expression for $s_{k,N}$ in \eqref{resca} (here $N=4$, $k=1$).} \label{fig001}
\end{figure}

Now take a general element $M\in G$ and $0<t\leq 1$, and define the function $F$ by
\begin{equation}\label{Fdef}
F(M,t)=\min\bigg\{ y > 0 \;\bigg|\; (x,y)\in\ZZ^2 M, \; -t < x \leq 1-t \bigg\} .
\end{equation}
To see the connection of $F$ with the gap $s_{k,N}$, define
\begin{equation}\label{lattice}
A_N=\begin{pmatrix} 1 & \alpha \\ 0 & 1 \end{pmatrix} \begin{pmatrix} N^{-1} & 0 \\ 0 & N \end{pmatrix} \in G,
\end{equation}
and note that, by rescaling the set in \eqref{resca}, we have
\begin{equation}
s_{k,N} = \frac1N \min\bigg\{ y > 0 \;\bigg|\; (x,y)\in\ZZ^2 A_N , \; -\frac{k}{N} < x \leq 1-\frac{k}{N} \bigg\}.
\end{equation}
Thus, 
\begin{equation}\label{key}
s_{k,N}= \frac1N \, F\bigg(A_N,\frac{k}{N}\bigg).
\end{equation}

We first check $F$ is well-defined as a function on the space of lattices $\GamG$ (Proposition \ref{prop1}), and then establish that the function $t\mapsto F(M,t)$ only takes at most three values for every fixed $M\in G$ (Proposition \ref{prop:three}). The latter implies the three gap theorem via \eqref{key}.

\begin{prop}\label{prop1}
$F$ is well-defined as a function $\GamG \times (0,1] \to \RR_{> 0}$.
\end{prop}

\begin{proof}
Let us begin by showing that
\begin{equation}\label{theset}
\bigg\{ y > 0 \;\bigg|\; (x,y)\in\ZZ^2 M,\; -t < x \leq 1-t  \bigg\} 
\end{equation}
is nonempty for every $M\in G$, $t\in(0,1]$. Let
\begin{equation}
M=\begin{pmatrix} a & b \\ c & d \end{pmatrix},
\end{equation}
and assume first that $a=0$. Then $c\neq 0$ and $b=-1/c$, and \eqref{theset} becomes
\begin{equation}\label{theset2}
\bigg\{ bm+dn > 0  \;\bigg|\; (m,n)\in\ZZ^2  ,\;  -t < cn \leq 1-t \bigg\} 
\supset |b|\NN,
\end{equation}
which is nonempty. If $a\neq 0$, we have
\begin{equation}
M= \begin{pmatrix} a & b \\ c & d \end{pmatrix} = \begin{pmatrix} a & 0 \\ c & a^{-1} \end{pmatrix} \begin{pmatrix} 1 & ba^{-1} \\ 0 & 1 \end{pmatrix} , 
\end{equation}
and so \eqref{theset} equals
\begin{equation}\label{theset3}
\bigg\{ y +ba^{-1} x > 0 \;\bigg|\;(x,y)\in\ZZ^2 \begin{pmatrix} a & 0 \\ c & a^{-1} \end{pmatrix} ,\;  -t < x \leq 1-t  \bigg\}  .
\end{equation}
Since $-t < x \leq 1-t$ implies $|x|\leq 1$, the set in \eqref{theset3} contains the set
\begin{multline}\label{theset4}
\bigg\{ y +ba^{-1} x  \;\bigg|\;(x,y)\in\ZZ^2 \begin{pmatrix} a & 0 \\ c & a^{-1} \end{pmatrix} ,\;  -t < x \leq 1-t,\; y  > |ba^{-1}|  \bigg\}  \\
= \bigg\{bm+dn  \;\bigg|\;(m,n)\in\ZZ^2 ,\;  -t < am+cn \leq 1-t,\;  n  > |b|  \bigg\}  .
\end{multline}
If $c/a$ is rational, there exist $(m,n)\in\ZZ^2 $ with $n> |b|$ such that $am+cn=0$. If $c/a$ is irrational, then the set $\{ am+cn \mid  (m,n)\in\ZZ^2  ,\; n  > |b| \}$ is dense in $\RR$. Therefore, in both cases, \eqref{theset4} is nonempty, and the minimum of \eqref{theset} exists due to the discreteness of $\ZZ^2 M$.

Finally, we note that $F(\,\cdot\,,t)$ is well-defined on $\GamG$ since $F(M,t)=F(\gamma M,t)$ for all $M\in G$, $\gamma\in\Gamma$.
\end{proof}

\begin{figure}
\begin{center}
\begin{minipage}{0.4\textwidth}
\unitlength0.1\textwidth
\begin{picture}(10,10)(0,0)
\put(0,1){\includegraphics[width=\textwidth]{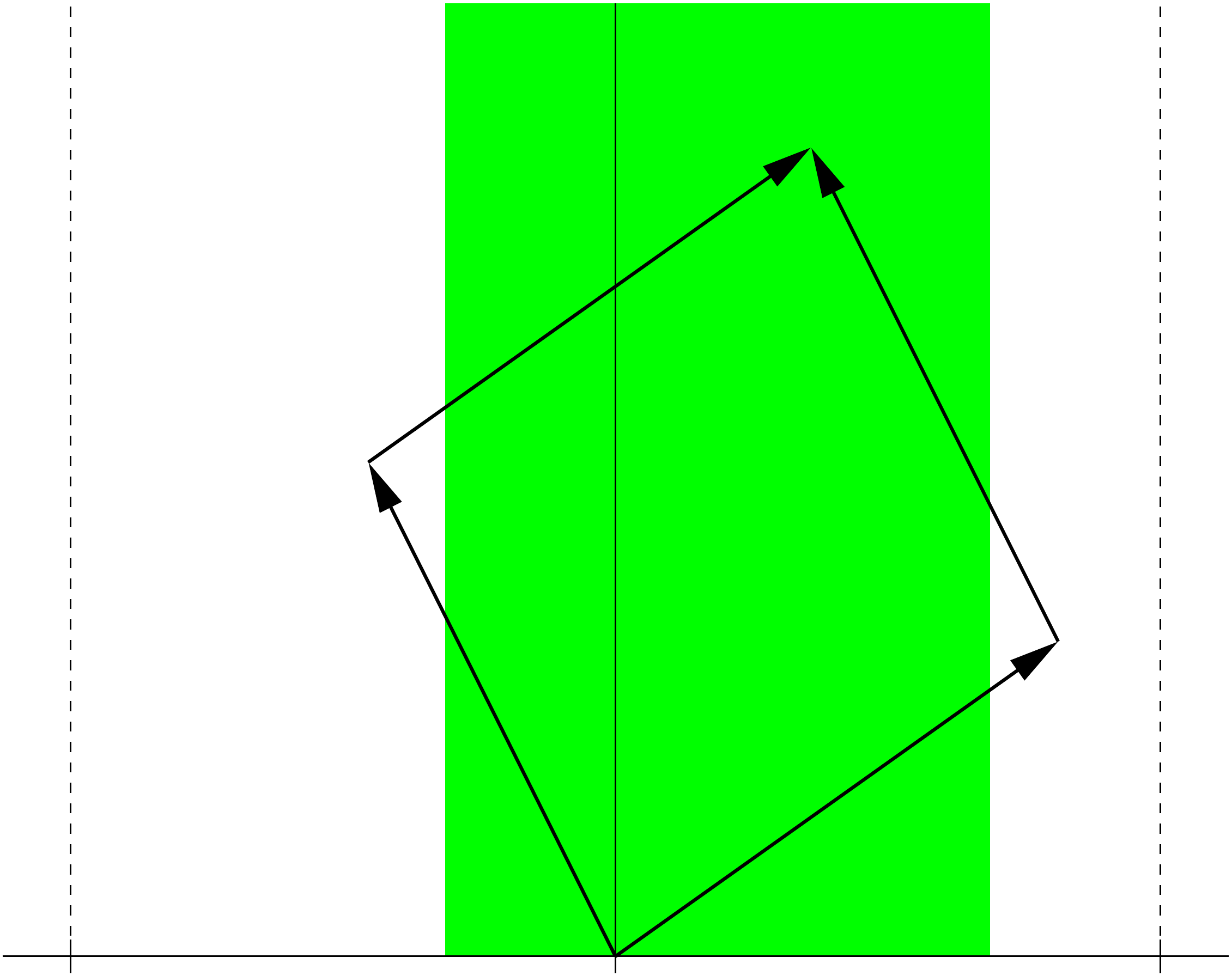}}
\put(2.5,5){$\vecs$}\put(8.8,3.5){$\vecr$}  \put(5.9,7.9){$\vecr+\vecs$}  
\put(0,0.5){$-1$}\put(3.2,0.5){$-t$} \put(4.85,0.5){$0$} \put(7.3,0.5){$1-t$} \put(9.3,0.5){$1$} 
\end{picture}
\end{minipage}
\end{center}
\caption{Illustration of the lattice configuration in the proof of Proposition \ref{prop:three}.} \label{fig1}
\end{figure}
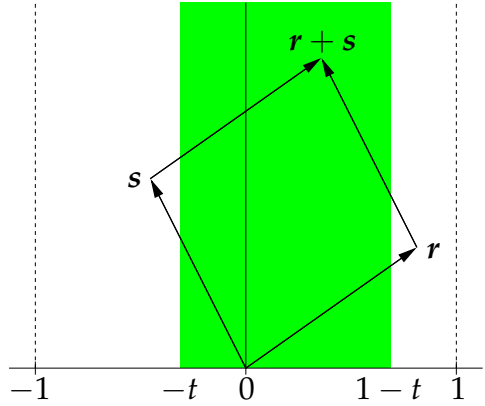

The following assertion implies the classical three gap theorem; recall \eqref{key}. 

\begin{prop}\label{prop:three}
For every given $M\in G$, the function $t\mapsto F(M,t)$ is piecewise constant and takes at most three distinct values. If there are three values, then the third is the sum of the first and second.
\end{prop}

\begin{proof}
Among all points of $\scrL=\ZZ^2 M$ in the region $\scrA=(-1,1)\times \RR_{>0}$, 
let $\vecr=(r_1,r_2)$ be a point with minimal second coordinate $r_2$. See Figure \ref{fig1}.
Next let $\vecs=(s_1,s_2)$ be a point in $\scrA\cap\scrL\setminus\Z\vecr$ with $s_2$ minimal. (If such a vector $\vecs$ does not exist, then $F(M,t)=r_2$ for all $t$.) Then $s_2\geq r_2>0$. Let us assume $s_2>r_2$ (the case $s_2=r_2$ is treated at the end of the proof).

The parallelogram $\vecnull,\vecr,\vecs,\vecr+\vecs$ does not contain any other lattice points: if $\vecu$ were such a lattice point, then $\vecu$ or $\vecr+\vecs-\vecu$ would have second coordinate smaller than $s_2$, contradicting the assumed minimality of $s_2$. This implies that $\vecr,\vecs$ form a basis of $\scrL$. 

Note that $r_1$ and $s_1$ must have opposite signs, i.e.\ $r_1s_1<0$,
since otherwise $\vecs-\vecr\in\scrA$ with a second coordinate that is smaller than $s_2$, 
contradicting the assumed minimality of $s_2$.
It follows that, if we set $\scrJ_{\vecr}=(0,1]\cap(-r_1,1-r_1]$
and $\scrJ_{\vecs}=(0,1]\cap(-s_1,1-s_1]$, 
then one of these intervals is of the form $(0,q]$ and the other is of the form $(q',1]$,
for some $q,q'\in(0,1)$. Note that both intervals are nonempty since $\vecr,\vecs\in\scrA$ by construction, and thus $|r_1|$, $|s_1|<1$. 
More explicitly,
\begin{equation}
\scrJ_{\vecr} = \begin{cases}
(-r_1,1] & \text{if $-1<r_1\leq 0$}\\
(0,1-r_1] & \text{if $0\leq r_1<1$,}
\end{cases}
\end{equation}
and similarly for $\scrJ_{\vecs}$.
Now in view of definition \eqref{Fdef}, we obtain
\begin{equation}
F(M,t)=
\begin{cases}
r_2 & \text{if $t\in\scrJ_{\vecr}$}\\
s_2 & \text{if $t\in\scrJ_{\vecs}\setminus\scrJ_{\vecr}$}\\
r_2+s_2 & \text{if $t\in(0,1]\setminus(\scrJ_{\vecr}\cup\scrJ_{\vecs}).$}
\end{cases}
\end{equation}
(Here the set $(0,1]\setminus(\scrJ_{\vecr}\cup\scrJ_{\vecs})$ may be empty.)
Thus, for any fixed $M$, the function $F(M,\,\cdot\,)$ can only take one of the three values $r_2,s_2,r_2+s_2$.

Now consider the remaining case $s_2=r_2$. We choose $\vecr,\vecs\in\scrA\cap\scrL$ so that $\vecr=(r_1,r_2)$ has minimal $r_1\geq 0$, and $\vecs=(s_1,r_2)$ has maximal $s_1< 0$. We can then proceed as above to obtain $F(M,t)=r_2$ for $t\in (0,1-r_1] \cup (-s_1,1]$ and  $F(M,t)=2 r_2$ for all other $t$ in $(0,1]$.

\end{proof}

\subsection*{Acknowledgment}
We thank both referees for their valuable feedback. The research leading to these results has received funding from the European Research Council under the European Union's Seventh Framework Programme (FP/2007-2013) / ERC Grant Agreement n. 291147. A.S.\ is supported by a grant from the G\"oran Gustafsson Foundation for Research in Natural Sciences and Medicine, and also by the Swedish Research Council Grant 621-2011-3629.

\end{document}